\documentclass{amsart}

\usepackage{amsmath,amstext,amsthm,amssymb,amsxtra}
\usepackage[colorlinks,citecolor=red,pagebackref,hypertexnames=false]{hyperref}

\newtheorem{theorem}{Theorem}[section]
\newtheorem*{theorem*}{Theorem}
\newtheorem{lemma}[theorem]{Lemma}
\newtheorem{corollary}[theorem]{Corollary}

\theoremstyle{remark}
\newtheorem{remark}[theorem]{Remark}

\theoremstyle{definition}
\newtheorem*{notation}{Notation}

\numberwithin{equation}{section}

\newcommand{\Pdn}{\mathcal{P}_{d,n}}
\newcommand{\Pdone}{\mathcal{P}_{d,1}}
\newcommand{\re}{\mathbb{R}}
\newcommand{\abs}[1]{\left\vert#1\right\vert}

\begin{document}

\title[Singular oscillatory integrals on $\re^n$]{Singular oscillatory integrals on $\re^n$}

\author[M. Papadimitrakis]{M. Papadimitrakis}
\email{papadim@math.uoc.gr}
\address{Department of Mathematics, University of Crete, Knossos Avenue 71409, Iraklio--Crete, Greece}

\author[I. R. Parissis]{I. R. Parissis}
\email{ioannis.parissis@gmail.com}
\address{Institutionen f\"or Matematik, Kungliga Tekniska H\"ogskolan, SE 100 44, Stockholm, SWEDEN.}
\thanks{}

\subjclass[2000]{Primary 42B20; Secondary 26D05}



\begin{abstract}Let $\Pdn$ denote the space of all real polynomials of degree at most $d$ on $\re^n$. We prove a new estimate for the logarithmic measure of the sublevel set of a polynomial $P\in \Pdone$. Using this estimate, we prove that
\begin{eqnarray*}
\sup_ {P\in\Pdn} \bigg|p.v.\int_{\re^n} {e^{iP(x)}\frac{\Omega(x/|x|)}{|x|^n}dx}\bigg|\leq c  \log d\,(\|\Omega\|_{L\log L(S^{n-1})}+1),
\end{eqnarray*} 
for some absolute positive constant c and every function $\Omega$ with zero mean value on the unit sphere $S^{n-1}$. This improves a result of Stein from \cite{S}.
\end{abstract}

\maketitle

\section{Introduction}
We denote by $\Pdn$ the vector space of all real polynomials of degree at most $d$ in $\re^n$. Let $K$ be a $-n$ homogeneous function on $\re^n$, that is, 
\begin{equation}
K(x)=\frac{\Omega(x/|x|) }{|x|^n},
\end{equation}
where $\Omega$ is some function on the unit sphere $S^{n-1}$. Consider the principal value integral
$$I_n(P)=\bigg|p.v.\int_{\re^n} e^{iP(x)}K(x)dx\bigg|.$$ 
Stein has proved in \cite{S} that if $\Omega$ has zero mean value on the unit sphere, then
\begin{equation}\label{eq1.2}
|I_n(P)|\leq c_d \|\Omega \|_{L^\infty(S^{n-1})},
\end{equation}
for some constant $c_d$ depending on $d$. We wish to obtain sharp estimates of the form \eqref{eq1.2}. The one dimensional analogue, namely the estimate
\begin{equation}
\bigg|p.v.\int_{\re} e^{iP(x)}\frac{dx}{x}\bigg| \leq c \log d,
\end{equation}
which was proved in \cite{P}, suggests that the constant $c_d$ in (\ref{eq1.2}) could be replaced by $c\log d$ for some absolute positive constant $c$. The fact that this is indeed the case is the content of the following theorem.
\begin{theorem}\label{main} Suppose that $K(x)=\Omega(x/|x|)/|x|^n$ where $\Omega$ has zero mean value on the unit sphere $S^{n-1}$. There exists an absolute positive constant $c$ such that
$$\sup_{P\in\Pdn}\bigg|p.v.\int_{\re^n} e^{iP(x)}K(x)dx\bigg|\leq c \log d \ (\|\Omega\|_{L\log L(S^{n-1})}+1).$$ 
\end{theorem}

\begin{remark}\label{odd} Suppose that $K(x)=\Omega(x/|x|)/|x|^n$ where the function $\Omega$ is odd on the unit sphere. It is an immediate consequence of the one-dimensional result that
$$\sup_{P\in\Pdn}\bigg|p.v.\int_{\re^n} e^{iP(x)}K(x)dx \bigg| \leq c \log d \  \|\Omega\|_{L^1 (S^{n-1})}$$
for some absolute positive constant $c$.
\end{remark}

The main ingredient of the proof of Theorem \ref{main} is an estimate for the logarithmic measure of the sublevel set of a real polynomial in one dimension. This is a lemma of independent interest which we now state.

\begin{lemma}[The logarithmic measure lemma]\label{logarithmic} Let $P(x)=\sum_{k=0} ^d b_k x^k$ be a real valued polynomial of degree at most $d$, $\alpha>0$ and $M=\max\{|b_k|:\frac{d}{2}< k \leq d\}$. If $E=\{x\geq1:|P(x)|\leq \alpha\}$, then
$$\int_E \frac{dx}{x} \leq c \min \bigg(\bigg(\frac{\alpha}{M}\bigg)^\frac{1}{d}, 1+\frac{1}{d}\log^+ \frac{\alpha}{M}\bigg),$$
where $c$ is an absolute positive constant.
\end{lemma}

Lemma \ref{logarithmic} should be compared to the following variation of a classical result of Vinogradov which can be found in \cite{V}:

\begin{lemma}\label{vinogradov} Let $P(x)=\sum_{k=0} ^d b_k x^k$ be a real valued polynomial of degree at most $d$, $\alpha>0$ and $ M_r=\max\{|b_k|:r\leq k \leq d\}$. Let $1<R$. Then
$$|\{x\in[1,R]:|P(x)|\leq \alpha \}| \leq c R^{1-\frac{r}{d}}  \frac{\alpha^\frac{1}{d}}{M_r ^\frac{1}{d}} ,$$
where $c$ is an absolute positive constant.
\end{lemma}

The estimates above depend on the length of the interval $[1,R]$ in all cases but the one where $r=d$. The dependence on $R$ is sharp as can be seen by a scaling argument.

When $r=d$ we get 

\begin{equation}\label{eq1.4}|\{x\in[1,R]:|P(x)|\leq \alpha \}| \leq c \frac{\alpha^\frac{1}{d}}{|b_d| ^\frac{1}{d}}.
\end{equation}

The last inequality corresponds to the following more general result about sublevel sets which was proved in \cite{AKC}:

\begin{lemma}\label{genphase} Let $\phi$ be a $C^k$ function on the interval $[1,R]$ for some $k\geq 1$ and $R>1$. Suppose that $|\phi^{(k)}(x)|\geq M$ on $[1,R]$. Then
$$|\{x\in[1,R]:|\phi(x)|\leq \alpha \}| \leq c k \frac{\alpha^\frac{1}{k}}{M ^\frac{1}{k}} ,$$
where $c$ is an absolute positive constant.
\end{lemma}

Observe that inequality (\ref{eq1.4}) can be deduced by Lemma \ref{genphase} by taking $k=d$ derivatives of the phase function $\phi(x)=P(x)$.

In case $n=1$ the ``linear" part $(\frac aM)^{\frac 1d}$ of the estimate of $ \int_E\frac 1x dx $ in Lemma \ref{logarithmic} is enough for the proof of Theorem \ref{main}. In fact, the author in \cite{P} used Lemma \ref{vinogradov} in some appropriate way to prove the above "linear" estimate of Lemma \ref{logarithmic}.

  In case $n>1$ the ``logarithmic" part of the estimate of $\int_E\frac 1x dx$ is essential in the proof of Theorem \ref{main} as can easily be seen by examining the argument therein.

     The structure of the rest of this work is as follows. In section \ref{s.prelim} we state some preliminary results. In section \ref{s.prooflog} we present the proof of Lemma \ref{logarithmic} and section \ref{s.proof} contains the proof of Theorem \ref{main}. Finally in section \ref{s.1d} we give a proof of Theorem \ref{main} in case $n=1$ which uses (the "linear" estimate in) Lemma \ref{logarithmic} and not Lemma \ref{vinogradov} and which is thus simpler than the proof appearing in \cite{P}.

\begin{notation} We will use the letter $c$ to denote an absolute positive constant which might change even in the same line of text.
\end{notation}
\section{Preliminary Results}\label{s.prelim}
As is usually the case when one deals with oscillatory integrals, a key Lemma is the classical van der Corput Lemma.
\begin{lemma}[van der Corput]\label{vander}
Let $\phi:[a,b]\rightarrow\re$ be a $C^1$ function and suppose that $|\phi^\prime(t)|\geq 1$ for all $t\in[a,b]$ and $\phi^{\prime}$ changes monotonicity $N$ times in $[a,b]$. Then, for every $\lambda \in \re$,

$$\bigg| \int _a ^b e^{i\lambda \phi(x)} dx \bigg | \leq \frac{cN}{|\lambda|}$$ 
where $c$ is an absolute constant independent of a,b and $\phi$.
\end{lemma}

The proof of Lemma \ref{vander} is a simple integration by parts.

We will also need a precise estimate for the Lebesgue measure of the sublevel set of a polynomial on $\re^n$.  
\begin{theorem}[Carbery,Wright]\label{carberakos} Suppose that $K\subset\re^n$ is a convex body of volume $1$ and $P\in\Pdn$. Let $1\leq q \leq \infty$. Then,
$$|\{x\in K:|P(x)|\leq \alpha\}|\leq c \min(qd,n) \alpha^\frac{1}{d} \|P\|_{L^q(K)} ^{-\frac{1}{d}}.$$
\end{theorem}
This is a consequence of a more general Theorem of Carbery and Wright and can be found in \cite{CW}.

\begin{corollary} \label{polweight}Let $P$ be a real homogeneous polynomial of degree $k$ on $\re^n$. Then
\begin{equation}
\int_{S^{n-1}}\frac{\|P\| ^\frac{1}{2k} _{L^\infty(S^{n-1})}}{|P(x^\prime)|^\frac{1}{2k}}d\sigma_{n-1}(x^\prime)\leq c.
\end{equation}
\end{corollary}
\begin{proof}[Proof of Corollary \ref{polweight}] Let $B=B(0,\rho)$ be the ball of volume $1$ on $\re^n$. For $\epsilon <\frac{1}{k}$ and some $\lambda>0$ to be defined later, we have
\begin{eqnarray*}
\int_B |P(x)|^{-\epsilon}dx&=& \int_0 ^\infty |\{x\in B:|P(x)|^{-\epsilon}\geq \alpha\}|d\alpha \\
&\leq& \lambda + \int_\lambda ^\infty |\{x\in B:|P(x)|< \alpha^{-\frac{1}{\epsilon}}\}|d\alpha\\
&\leq& \lambda +c n \|P\|_{L^\infty(B)} ^{-\frac{1}{k}} \frac{\lambda ^{-\frac{1}{k\epsilon}+1}}{\frac{1}{k\epsilon}-1},
\end{eqnarray*}
using Theorem \ref{carberakos}. Optimizing in $\lambda$ we get
$$\int_B |P(x)|^{-\epsilon}dx\leq \bigg(cn\frac{k\epsilon}{1-k\epsilon}\bigg)^{k\epsilon}\|P\|_{L^\infty(B)} ^{-\epsilon} .$$
Using polar coordinates and setting $\epsilon=\frac{1}{2k}<\frac{1}{k}$, we then get
\begin{eqnarray*}
\|P\|_{L^\infty(S^{n-1})} ^\frac{1}{2k} \int_{S^{n-1}} |P(x^\prime)|^{-\frac{1}{2k}}d\sigma_{n-1}(x^\prime)  &\leq& c \frac{n^\frac{3}{2}}{\rho^{n}}=c\frac{n^\frac{3}{2}\pi^\frac{n}{2}}{\Gamma(\frac{n}{2}+1)}\\
&\leq& c\frac{n^\frac{3}{2}(e\pi)^\frac{n}{2}}{(\frac{n}{2}+1)^\frac{n+1}{2}}\leq c,
\end{eqnarray*}
which completes the proof.
\end{proof}

\section{The logarithmic measure lemma}\label{s.prooflog}
The proof of Lemma \ref{logarithmic} is motivated by an argument of Vinogradov from \cite{V}, used to estimate the \emph{Lebesgue} measure of the sublevel set of a polynomial in a bounded interval. We fix a polynomial $P(x)=\sum_{k=0} ^d b_k x^k$ and look at the set $E=\{x\geq1:|P(x)|\leq \alpha\}$.  Note that by replacing $\alpha$ with $\alpha M$ in the statement of the lemma, it is enough to consider the case $M=1$. Since $E$ is a closed set we can find points $x_0,x_1,\ldots,x_d\in E$ such that $x_0<x_1<\cdots<x_d$ and
$$\frac{1}{d}\int_E \frac{dx}{x}=\int_{E\cap[x_j,x_{j+1}]}\frac{dx}{x}\leq \log \frac{x_{j+1}}{x_j}, \ \ \ \ \ 0\leq j \leq d-1.$$
We set $\mu=\int_E\frac{dx}{x}$ and $t=e^\frac{\mu}{d}>1$ and we have that $x_{j+1}\geq t x_j$, $0\leq j \leq d-1$. The Lagrange interpolation formula is
$$P(x)=\sum_{j=0} ^d P(x_j)\frac{(x-x_0)\cdots(\widehat{x-x_j})\cdots(x-x_d)}{(x_j-x_0)\cdots(\widehat{x_j-x_j})\cdots(x_j-x_d)}, \ x\in\re,$$
where $\hat{u}$ means that $u$ is omitted. Thus,
$$b_k=\sum_{j=0} ^d P(x_j)(-1)^{d-k}\frac{\sigma_{d-k}(x_0,\ldots,\widehat{x_j},\ldots,x_d)}{(x_j-x_0)\cdots(\widehat{x_j-x_j})\cdots(x_j-x_d)},$$
where $\sigma_l$ is the $l$-th elementary symmetric function of its variables. Therefore
\begin{eqnarray*}
|b_k|&\leq& \alpha \sum_{j=0} ^d \frac{\sigma_{d-k}(x_0,\ldots,\widehat{x_j},\ldots,x_d)}{|x_j-x_0|\cdots|\widehat{x_j-x_j}|\cdots|x_j-x_d|}\\&=&\alpha \sum_{j=0} ^d \frac{\sigma_{k}(\frac{1}{x_0},\ldots,\widehat{\frac{1}{x_j}},\ldots,\frac{1}{x_d})}{(\frac{x_j}{x_0}-1)\cdots(\frac{x_j}{x_{j-1}}-1)(1-\frac{x_j}{x_{j+1}})\cdots(1-\frac{x_j}{x_d})}\\
&\leq&\alpha \sum_{j=0} ^d \frac{\sigma_{k}(1,\ldots,\widehat{\frac{1}{t}},\ldots,\frac{1}{t^d})}{(t^j-1)\cdots(t-1)(1-\frac{1}{t})\cdots(1-\frac{1}{t^{d-j}})}.
\end{eqnarray*}
It is easy to see that there exists precisely one $j$, $0\leq j\leq\frac{d-1}{2}<d$, for which
\begin{equation}\label{eq3.1}
t^{j-1}<\frac{2t^d}{t^{d+1}+1}\leq t^j.
\end{equation}
It is exactly for this $j$ that $(t^j-1)\cdots(t-1)(1-\frac{1}{t})\cdots(1-\frac{1}{t^{d-j}})$ takes its minimum value as $j$ runs from $0$ to $d$. On the other hand we have
$$\sum_{j=0} ^d \sigma_k\bigg(1,\ldots,\widehat{\frac{1}{t_j}},\ldots,\frac{1}{t^k}\bigg)=(d+1-k)\sigma_k\bigg(1,\ldots,\frac{1}{t^d}\bigg)$$
and, hence
\begin{align}\label{eq3.2}
\nonumber|b_k|&\leq& \alpha \  (d+1-k)\sigma_k\bigg(1,\ldots,\frac{1}{t^d}\bigg)\frac{1}{(t^j-1)\cdots(t-1)(1-\frac{1}{t})\cdots(1-\frac{1}{t^{d-j}})}\\ \\
&\leq&\frac{\alpha \ (d+1-k)\binom{d+1}{k}}{1\cdot t\cdots t^k}\frac{1}{(t^j-1)\cdots(t-1)(1-\frac{1}{t})\cdots(1-\frac{1}{t^{d-j}})}.\nonumber
\end{align}
From (\ref{eq3.1}) we easily see that $t^j<2$ and, since $\frac{\log(x-1)}{x}$ is increasing in the interval $(1,2)$, we find
\begin{eqnarray}\label{eq3.3}
&&\nonumber\log(t-1)+\cdots+\log(t^j-1)\\&=&\frac{t}{t-1}\bigg(\frac{\log(t-1)}{t}(t-1)+\cdots+\frac{\log(t^j-1)}{t^j} (t^j-t^{j-1})\bigg)  \\ \nonumber&\geq& \frac{t}{t-1}\int_1 ^{t^j}\frac{\log(x-1)}{x}dx=\frac{t}{t-1}\int_0 ^{t^{j}-1}\frac{\log x}{1+x}dx.
\end{eqnarray}
Similarly, since $\frac{\log(1-x)}{x}$ is decreasing in the interval $(0,1)$ we get
\begin{eqnarray}\label{eq3.4}
&&\nonumber\log\bigg(1-\frac{1}{t^{d-j}}\bigg)+\cdots+\log\bigg(1-\frac{1}{t}\bigg)\\&=&\frac{1}{t-1}\bigg(\frac{\log(1-\frac{1}{t^{d-j}})}{\frac{1}{t^{d-j}}}\bigg(\frac{1}{t^{d-j-1}}-\frac{1}{t^{d-j}}\bigg)+\cdots+\frac{\log(1-\frac{1}{t})}{\frac{1}{t}} \bigg(1-\frac{1}{t}\bigg)\bigg) \\ \nonumber&\geq& \frac{1}{t-1}\int_{\frac{1}{t^{d-j}}} ^1\frac{\log(1-x)}{x}dx=\frac{1}{t-1}\int_0 ^{1-\frac{1}{t^{d-j}}}\frac{\log x}{1-x}dx.
\end{eqnarray}
We let 
$$A=\frac{t^d-1}{t^d+1}, \ \ B=t^j-1, \  \ \Gamma=1-\frac{1}{t^{d-j}},$$
and, obviously, $0<A, B,\Gamma<1$. From (\ref{eq3.3}) and (\ref{eq3.4}) we have
\begin{eqnarray*}
&& \log(t-1)+\cdots+\log(t^j-1)+\log\bigg(1-\frac{1}{t^{d-j}}\bigg)+\cdots+\log\bigg(1-\frac{1}{t}\bigg)\\
&\geq&\frac{t}{t-1}\int_0 ^{t^{j}-1}\frac{\log x}{1+x}dx+\frac{1}{t-1}\int_0 ^{1-\frac{1}{t^{d-j}}}\frac{\log x}{1-x}dx\\
&=&\frac{t}{t-1}\int_0 ^B\frac{\log x}{1+x}dx+\frac{1}{t-1}\int_0 ^\Gamma \frac{\log x}{1-x}dx\\
&=&-\frac{t}{t-1}B\log\frac{1}{B}-\frac{1}{t-1}\Gamma \log \frac{1}{\Gamma} - O\bigg(\frac{t}{t-1}B\bigg)-O\bigg(\frac{1}{t-1}\Gamma\bigg).
\end{eqnarray*}
From (\ref{eq3.1}) we get $B,\Gamma\leq\frac{t^{d+1}-1}{t^{d+1}+1}$ and, since
$\frac{t+1}{t-1}\frac{t^{d+1}-1}{t^{d+1}+1}$ is decreasing in $t\in(1,+\infty)$, we find
$$\frac{t}{t-1}B\leq \frac{t+1}{t-1}\frac{t^{d+1}-1}{t^{d+1}+1}\leq d+1$$
and, similarly,
$$\frac{1}{t-1}\Gamma \leq \frac{t+1}{t-1}\frac{t^{d+1}-1}{t^{d+1}+1}\leq d+1.$$
Therefore
\begin{eqnarray*}
&& \log(t-1)+\cdots+\log(t^j-1)+\log\bigg(1-\frac{1}{t^{d-j}}\bigg)+\cdots+\log\bigg(1-\frac{1}{t}\bigg) \\
&\geq& -\frac{t}{t-1}B\log\frac{1}{B}-\frac{1}{t-1}\Gamma \log \frac{1}{\Gamma} - cd \\
&\geq& -\frac{2}{t-1}A\log\frac{1}{A}-\frac{1}{t-1}\bigg(B\log\frac{1}{B}+\Gamma \log \frac{1}{\Gamma}-2A\log\frac{1}{A}\bigg)-cd.
\end{eqnarray*}
Now
\begin{eqnarray*}
B\log\frac{1}{B}+\Gamma \log \frac{1}{\Gamma}-2A\log\frac{1}{A}&=&(B+\Gamma	-2A)\log\frac{1}{A}+A\frac{B}{A}\log\frac{A}{B}+A\frac{\Gamma}{A}\log\frac{A}{\Gamma}\\&\leq&\ \bigg(\frac{B+\Gamma}{A}-2\bigg)A\log\frac{1}{A}+cA.
\end{eqnarray*}
Using (\ref{eq3.1})
$$\frac{B+\Gamma}{A}-1\leq\frac{2(t-1)}{t^{d+1}+1}$$
and we conclude that
\begin{eqnarray*}
\frac{1}{t-1}\bigg(B\log\frac{1}{B}+\Gamma \log \frac{1}{\Gamma}-2A\log\frac{1}{A}\bigg) &\leq& 
\frac{2}{t^{d+1}+1}A\log\frac{1}{A}+\frac{c}{t-1}A \\ 
&\leq& c+c\frac{t+1}{t-1}\frac{t^d-1}{t^d+1}\leq cd.
\end{eqnarray*}
Therefore
\begin{eqnarray*}
 &&\log(t-1)+\cdots+\log(t^j-1)+\log\bigg(1-\frac{1}{t^{d-j}}\bigg)+\cdots+\log\bigg(1-\frac{1}{t}\bigg)\\&\geq &
-\frac{2}{t-1}A\log\frac{1}{A} -cd
\end{eqnarray*}
and, finally, (\ref{eq3.2}) implies that for some $k>\frac{d}{2}$
$$1\leq \frac{c_o^d\alpha}{t^\frac{k(k-1)}{2}}\bigg(\frac{1}{A}\bigg)^\frac{2A}{t-1},$$
where $c_o$ is an absolute positive constant.

\noindent\textbf{case 1:} $c_o \alpha^\frac{1}{d}<\frac{1}{2}$. Then, since $\frac{2A}{t-1}\leq \frac{t+1}{t-1}A\leq d$, we get
$$A^d\leq A^\frac{2A}{t-1}\leq c_o ^d\alpha $$
which implies
$$ \frac{t^d-1}{t^d+1}=A \leq c_o \ \alpha^\frac{1}{d}$$
and, finally,
$$\mu\leq e^\mu-1=t^d-1\leq 4c_o\alpha^\frac{1}{d}.$$

\noindent\textbf{case 2:} $c_o \alpha^\frac{1}{d}\geq \frac{1}{2}$, $t^d<2$. Then 
$$1<e^\mu=t^d<4c_o\alpha^\frac{1}{d}$$
which shows that
$$\mu<\log^+(4c_o)+\frac{\log^+\alpha}{d}.$$

\noindent\textbf{case 3:} $c_o \alpha^\frac{1}{d}\geq\frac{1}{2}$, $t^d \geq 2$. Then $A\geq \frac{1}{3}$ and $\frac{2A}{t-1}\leq \frac{t+1}{t-1}A\leq d$ and, hence,
$$\frac{1}{3^d}t^\frac{k(k-1)}{2}\leq c_o^d \alpha.$$
We conclude that 
$$\mu\leq\frac{2d^2}{k(k-1)}\bigg(\log^+(3c_o)+\frac{\log^+\alpha}{d}\bigg)\leq c\bigg(1+\frac {\log^+ \alpha}{d}\bigg)$$
since $k>\frac{d}{2}$. 
fsection{Proof of Theorem \ref{main}}\label{s.proof}

Let $\Omega$ be a function with zero mean value on the unit sphere $S^{n-1}$ belonging to the class $L\log L(S^{n-1})$, that is
$$\|\Omega\|_{L\log L(S^{n-1})}=\int_{S^{n-1}}|\Omega(x^\prime)|(1+\log^+|\Omega(x^\prime)|)d\sigma_{n-1}(x^\prime)<\infty.$$
Set $K(x)=\Omega(x/|x|)/|x|^n$ and let $P\in\Pdn$. We will show the theorem for $d=2^m$, for some $m\geq 0$. The general case is then an immediate consequence.

We set
$$C_d=\sup_{\substack{0<\epsilon<R \\ P\in \Pdn}}\abs{\int_{\epsilon\leq |x|\leq R}e^{iP(x)}K(x)dx},$$
where $C_d$ is a constant depending on $d$, $\Omega$ and $n$.
For $0<\epsilon<R$ and $P\in\Pdn$ we write,

$$I_{\epsilon,R}(P)=\int_{\epsilon\leq |x|\leq R}e^{iP(x)}K(x)dx = \int_{S^{n-1}} \int _\epsilon ^R e^{iP(rx^\prime)}\frac{dr}{r} \Omega(x^\prime) d\sigma_{n-1}(x^\prime).$$
For $x^\prime\in S^{n-1}$, we have that $P(rx^\prime)=\sum_{j=1} ^d P_j(x^\prime) r^j$ where $P_j$ is a homogeneous polynomial of degree $j$. Observe that we can omit the constant term, without loss of generality. Set also $m_j=\|P_j\|_{L^\infty(S^{n-1})}$. Since $\epsilon$ and $R$ are arbitrary positive numbers, by a dilation in $r$ we can assume that $\max_{\frac{d}{2}<j\leq d}m_j= 1$ and, in particular, that $m_{j_o}=1$ for some $\frac{d}{2}<j_o\leq d$. We also write $Q(x)=\sum_{j=1} ^\frac{d}{2}P_j(x)$. We split the integral in two parts as follows
\begin{eqnarray*}
|I_{\epsilon,R}(P)|&\leq& \abs{\int_{S^{n-1}}\int_\epsilon ^1 e^{iP(rx^\prime)}\frac{dr}{r}\Omega(x^\prime)d\sigma_{n-1}(x^\prime)}\\&+& \abs{\int_{S^{n-1}}\int_1 ^Re^{iP(rx^\prime)}\frac{dr}{r}\Omega(x^\prime)d\sigma_{n-1}(x^\prime)}=I_1+I_2.
\end{eqnarray*}
For $I_1$ we have that
\begin{eqnarray*}
I_1 &\leq & \int_{S^{n-1}}\int_0^1\abs{ e^{iP(rx^\prime)}-e^{iQ(rx^\prime)}}\frac{dr}{r}|\Omega(x^\prime)|d\sigma_{n-1}(x^\prime) \\
&+& \abs{\int_{S^{n-1}}\int_\epsilon ^1 e^{iQ(rx^\prime)}\frac{dr}{r}\Omega(x^\prime)d\sigma_{n-1}(x^\prime)}\\
&\leq & \sum_{\frac{d}{2}<j\leq d} \frac{m_j}{j} \|\Omega\|_{L^1(S^{n-1})}+C_{\frac{d}{2}} \leq c\|\Omega\|_{L^1(S^{n-1})}+C_{\frac{d}{2}}.
\end{eqnarray*}

For $I_2$ we write
\begin{eqnarray*}
I_2&\leq&\int_{S^{n-1}}\abs{\int _{\{r\in[1,R]:|\frac{\partial P(rx^\prime)}{\partial r}|>d\}}e^{iP(rx^\prime)}\frac{dr}{r} }|\Omega(x^\prime)|d\sigma_{n-1}(x^\prime)\\&+&\int_{S^{n-1}}\int _{\{r\in[1,R]:|\frac{\partial P(rx^\prime)}{\partial r}|\leq d\}}\frac{dr}{r} |\Omega(x^\prime)|d\sigma_{n-1}(x^\prime).
\end{eqnarray*}
Since $\{r\in[1,R]:|\frac{\partial P(rx^\prime)}{\partial r}|>d\}$ consists of at most $O(d)$ intervals where $\frac{\partial P(rx^\prime)}{\partial r}$ is monotonic, a simple corollary to van der Corput's lemma for the first derivative \cite[corollary on p. 334]{S1} gives the bound
$$\int_{S^{n-1}}\abs{\int _{\{r\in[1,R]:|\frac{\partial P(rx^\prime)}{\partial r}|>d\}}e^{iP(rx^\prime)}\frac{dr}{r} }|\Omega(x^\prime)|d\sigma_{n-1}(x^\prime) \leq c \|\Omega\|_{L^1(S^{n-1})}.$$
On the other hand, the logarithmic measure lemma implies that
\begin{eqnarray*}
&&\int_{S^{n-1}}\int _{\{r\in[1,R]:|\frac{\partial P(rx^\prime)}{\partial r}|\leq d\}}\frac{dr}{r} |\Omega(x^\prime)|d\sigma_{n-1}(x^\prime) \\ &\leq& c\|\Omega\|_{L^1(S^{n-1})}+c\frac{1}{d}\int_{S^{n-1}} \log\frac{d}{\max_{\frac{d}{2}<j\leq d}\{j|P_j(x^\prime)|\}} |\Omega(x^\prime)|d\sigma_{n-1}(x^\prime).
\end{eqnarray*}
Combining the estimates we get
$$C_d\leq c\|\Omega\|_{L^1(S^{n-1})}+C_{\frac{d}{2}}+c\frac{2j_o}{d}\int_{S^{n-1}}\
\log \frac{\|P_{j_o}\|^ \frac{1}{2j_o} _{L^\infty(S^{n-1})}}{|P_{j_o}(x^\prime)|^\frac{1}{2j_o}}|\Omega(x^\prime)|d\sigma_{n-1}(x^\prime)$$
and, from  Young's inequality,
\begin{eqnarray*}
 C_d\leq c\|\Omega\|_{L^1(S^{n-1})}+C_{\frac{d}{2}}&+&c\int_{S^{n-1}}\
 \frac{\|P_{j_o}\|^ \frac{1}{2j_o} _{L^\infty(S^{n-1})}}{|P_{j_o}(x^\prime)|^\frac{1}{2j_o}}d\sigma_{n-1}(x^\prime)\\&+&c\int_{S^{n-1}}|\Omega(x^\prime)|(1+\log^+|\Omega(x^\prime)|)d\sigma_{n-1}(x^\prime).
 \end{eqnarray*}
Now, using corollary \ref{polweight} we get 
$$C_d\leq C_{\frac{d}{2}}+c(\|\Omega\|_{L\log L(S^{n-1})}+1).$$
Since $d=2^m$, this means that
$$C_{2^m}\leq C_{2^{m-1}}+c(\|\Omega\|_{L\log L(S^{n-1})}+1).$$
Using induction on $m$ we get that $C_{2^m}\leq C_1+cm(\|\Omega\|_{L\log L(S^{n-1})}+1)$. Observe that $C_1$ corresponds to some polynomial $P(x)=b_1x_1+\cdots + b_n x_n$. We write
\begin{eqnarray*}
&&\abs{\int_{\epsilon<|x|<R}e^{iP(x)}K(x)dx}=\\
&&=\abs{\int_{S^{n-1}}\int_{\epsilon} ^R \{e^{irP(x^\prime)}-e^{ir\|P\|_{L^\infty(S^{n-1})}}\}\frac{dr}{r}\Omega(x^\prime)d\sigma_{n-1}(x^\prime)}.
\end{eqnarray*}
Using the simple estimate $$\abs{\int_\epsilon ^R \{e^{iar}-e^{ibr}\}\frac{dr}{r}}\leq c+c\abs{\log\abs{\frac{b}{a}}}$$
\\
we get 
\begin{eqnarray*}
\abs{\int_{\epsilon<|x|<R}e^{iP(x)}K(x)dx}&\leq& c\|\Omega\|_{L^1(S^{n-1})}+\\&&+c\int_{S^{n-1}}\log\frac{\|P\|^ \frac{1}{2} _{L^\infty(S^{n-1})}}{|P(x^\prime)|^\frac{1}{2}}|\Omega(x^\prime)|d\sigma_{n-1}(x^\prime).
\end{eqnarray*}
Hence,  $C_1\leq c\|\Omega\|_{L^1(S^{n-1})}+c+\|\Omega\|_{L\log L(S^{n-1})}$ 
and $$C_{2^m}\leq c m(\|\Omega\|_{L\log L(S^{n-1})}+1).	$$ 
The case of general d is now trivial. If $2^{m-1}<d\leq 2^m$ then $$C_d\leq C_{2^m}\leq c m(\|\Omega\|_{L\log L(S^{n-1})}+1)\leq c \log d (\|\Omega\|_{L\log L(S^{n-1})}+1).$$
\section{The one dimensional case revisited}\label{s.1d}
We will attempt to give a short proof of the one dimensional analogue of theorem \ref{main}. This is a slight simplification of the proof in \cite{P}, with the aid of the logarithmic measure lemma. 

So, fix a real polynomial $P(x)=b_0+b_1x+\cdots+b_dx^d$ and consider the quantity
$$C_d=\sup_{0<\epsilon<R} \abs{\int_{\epsilon<|x|<R} e^{iP(x)}\frac{dx}{x}}.$$
By the same considerations as in the $n-$dimensional case, we can assume that $P$ has no constant term and that it can be decomposed in the form
$$P(x)=\sum_{0<j\leq \frac{d}{2}} b_j x^j+\sum_{\frac{d}{2}<j\leq d} b_j x^j=Q(x)+R(x),$$
where $\max_{\frac d2<j\leq d}|b_j|=1$. As a result
\begin{eqnarray*}
\abs{\int_{\epsilon<|x|<R} e^{iP(x)}\frac{dx}{x}}&\leq& C_{\frac{d}{2}}+\int_{0<|x|<1}\frac{|R(x)|}{x}dx +\abs{\int_{1<|x|<R} e^{iP(x)}\frac{dx}{x}}\\
&\leq&C_{\frac{d}{2}}+c+I.
\end{eqnarray*}
We split $I$ as follows
\begin{eqnarray*}
I\leq \abs{\int_{\{x\in[1,R):|P^\prime(x)|>d\}}e^{iP(x)}\frac{dx}{x}}+\int_{\{x\geq1:|P^\prime(x)|\leq d\}}\frac{dx}{x}.
\end{eqnarray*}
Now, using Proposition \ref{vander} for the first summand in the above estimate and the logarithmic measure lemma to estimate the second summand, we get that $I\leq c$. But this means that $C_d\leq C_{\frac{d}{2}}+c$ which completes the proof by considering first the case $d=2^m$ for some $m$, as in the $n-$dimensional case.

\bibliographystyle{amsplain}

\bibliography{logarithmic_bib}

\providecommand{\bysame}{\leavevmode\hbox to3em{\hrulefill}\thinspace}
\providecommand{\MR}{\relax\ifhmode\unskip\space\fi MR }
\providecommand{\MRhref}[2]{%
  \href{http://www.ams.org/mathscinet-getitem?mr=#1}{#2}
}
\providecommand{\href}[2]{#2}
\begin{thebibliography}{1}

\bibitem{AKC}
G.~I. Arkhipov, A.~A. Karacuba, and V.~N. {\v{C}}ubarikov, \emph{Trigonometric
  integrals}, Izv. Akad. Nauk SSSR Ser. Mat. \textbf{43} (1979), no.~5,
  971--1003, 1197. \MR{MR552548 (81f:10050)}

\bibitem{CW}
A.~Carbery and J.~Wright, \emph{Distributional and {$L\sp q$} norm inequalities
  for polynomials over convex bodies in {$\mathbb R\sp n$}}, Math. Res. Lett.
  \textbf{8} (2001), no.~3, 233--248. \MR{MR1839474 (2002h:26033)}

\bibitem{P}
I.~R. Parissis, \emph{A sharp bound for the {S}tein-{W}ainger oscillatory
  integral}, Proc. Amer. Math. Soc. \textbf{136} (2008), no.~3, 963--972
  (electronic). \MR{MR2361870}

\bibitem{S}
E.~M. Stein, \emph{Oscillatory integrals in {F}ourier analysis}, Beijing
  lectures in harmonic analysis (Beijing, 1984), Ann. of Math. Stud., vol. 112,
  Princeton Univ. Press, Princeton, NJ, 1986, pp.~307--355. \MR{MR864375
  (88g:42022)}

\bibitem{S1}
Elias~M. Stein, \emph{Harmonic analysis: real-variable methods, orthogonality,
  and oscillatory integrals}, Princeton Mathematical Series, vol.~43, Princeton
  University Press, Princeton, NJ, 1993, With the assistance of Timothy S.
  Murphy, Monographs in Harmonic Analysis, III. \MR{MR1232192 (95c:42002)}

\bibitem{V}
I.~M. Vinogradov, \emph{Selected works}, Springer-Verlag, Berlin, 1985, With a
  biography by K. K. Mardzhanishvili, Translated from the Russian by Naidu Psv
  [P. S. V. Naidu], Translation edited by Yu.\ A. Bakhturin. \MR{MR807530
  (87a:01042)}

\end{thebibliography}

\end{document}